\documentclass{amsart}
\usepackage{amscd,amsmath,amssymb,amsthm,bm,cases,color,comment,mathrsfs,mathtools}
\usepackage{graphicx}
\usepackage[all]{xy}

\usepackage{comment}
\usepackage{combelow}

\makeatletter
\@namedef{subjclassname@2020}{%
\textup{2020} Mathematics Subject Classification}
\makeatother
\theoremstyle{definition}
\newtheorem{definition}{Definition}[section]
\newtheorem{remark}[definition]{Remark}

\theoremstyle{plain}
\newtheorem{theorem}[definition]{Theorem}
\newtheorem{proposition}[definition]{Proposition}
\newtheorem{lemma}[definition]{Lemma}
\newtheorem{corollary}[definition]{Corollary}
\numberwithin{equation}{section}

\newcommand{\qisim}{\mathrel{\underset{{\rm q.i.}}{\sim}}}

\title[Quasimorphisms on $\mathrm{Diff}_0(N_g)$]{Quasimorphisms on nonorientable surface diffeomorphism groups}

\author[M.~Kimura]{Mitsuaki Kimura}
\address[Mitsuaki Kimura]{Department of Mathematics, Osaka Dental University, 8-1 Kuzuha Hanazono-cho, Hirakata, Osaka 573-1121, Japan}
\email{kimura-m@cc.osaka-dent.ac.jp}

\author[E.~Kuno]{Erika Kuno}
\address[Erika Kuno]{Department of Mathematics,
Graduate School of Science,
Osaka University,
1-1 Machikaneyama-cho Toyonaka, Osaka 560-0043, Japan
}
\email{e-kuno@math.sci.osaka-u.ac.jp}

\date{}
\begin{document}
\begin{abstract}
Bowden, Hensel, and Webb constructed infinitely many quasimorphisms on the diffeomorphism groups of orientable surfaces. In this paper, we extend their result to nonorientable surfaces. 
Namely, we prove that the space of nontrivial quasimorphisms $\widetilde{QH}(\mathrm{Diff}_0(N_g))$ on the identity component of the diffeomorphism group $\mathrm{Diff}_0(N_g)$ on a closed nonorientable surface $N_g$ of genus $g\geq 3$ is infinite-dimensional. 
As a corollary, we obtain the unboundedness of the commutator length and the fragmentation length on $\mathrm{Diff}_0(N_g)$.
\end{abstract}
\maketitle

\section{Introduction} \label{Intro}

Geometric group theory studies infinite groups from a geometric point of view.
It has mainly focused on finite generated groups, but recently its method has also been applied to the study of diffeomorphism groups.

Let $\mathrm{Diff}_0(S)$ denote the group of diffeomorphisms on a closed orientable surface $S$ which are isotopic to the identity.
In \cite{Bowden--Hensel--Webb22}, Bowden, Hensel, and Webb introduced 
the \textit{fine curve graph} $C^{\dagger}(S)$ of $S$
and proved that the graph $C^{\dagger}(S)$ is Gromov hyperbolic and the action of the group $\mathrm{Diff}_0(S)$ on $C^{\dagger}(S)$ satisfies the condition of Bestvina--Fujiwara \cite{Bestvina--Fujiwara07} (see Proposition \ref{prop:BF}).
As a consequence, the group $\mathrm{Diff}_0(S_g)$ admits infinitely many nontrivial quasimorphisms for every closed orientable surface $S_g$ of genus $g\geq 1$.

In this paper, we extend the result of Bowden, Hensel and Webb to every closed nonorientable surface $N_{g}$ of genus $g \geq 3$.

\begin{theorem} \label{thm:main}
  For $g\geq 3$, the space of nontrivial quasimorphisms $\widetilde{QH}(\mathrm{Diff}_0(N_g))$ on $\mathrm{Diff}_0(N_g)$ is infinite-dimensional.
\end{theorem}

For the definition of quasimorphism, see Section \ref{subsec:qm}.
As in \cite{Bowden--Hensel--Webb22}, 
we consider the fine curve graph $C^{\dagger}(N)$ for a nonorientable surface $N$ and prove the following:
\begin{itemize}
    \item The graph $C^{\dagger}(N)$ is Gromov-hyperbolic (Theorem \ref{unif_hyp_fine_curve_graph}).
    \item The action of $\mathrm{Diff}_0(N_g)$ on $C^{\dagger}(N)$ satisfies the condition of Bestvina--Fujiwara.
\end{itemize}
 To prove 
 the former, we introduce the notion of the extended fine curve graph $\mathcal{C}^{\pm \dagger}(N)$ and the extended two-sided fine curve graph
$\mathcal{C}^{\pm\dagger}_{\mathrm{two}}(N)$, and prove that they are quasi-isometric to $C^{\dagger}(N)$ (Proposition~\ref{fine_curve_graphs_are_quasi-isometric}).
To observe the latter, 
we will use the theorem of Ichihara and Motegi \cite{{Ichihara--Motegi05}} (Theorem \ref{thm:IM}) instead of Kra's theorem \cite{Kra81} (Theorem \ref{thm:Kra}), which was used in the case of orientable surfaces \cite{Bowden--Hensel--Webb22}.


We discuss an application of Theorem \ref{thm:main}.
It is known that quasimorphisms are useful for studying the large scale geometry of groups. 

For $x \in [G,G]$, the \textit{commutator length} ${\rm cl}(x)$ of $x$
is the least number of commutators in $G$ whose product is equal to $x$. 
For a homogeneous quasimorphism $\phi$, we have $|\phi(x)| \leq 2D(\phi) \cdot {\rm cl}(x) $ by elementary calculations. Thus, ${\rm cl}$ is unbounded if $G$ admits a nontrivial quasimorphism.
Since ${\rm Diff}_0(M)$ is a perfect group \cite{Thurston74} (see also \cite{Mann16}), the commutator length ${\rm cl}$ is defined on the whole group ${\rm Diff}_0(M)$.


Diffeomorphism groups also admits a natural norm called the \textit{fragmentation norm}.
For a closed manifold $M$, any diffeomorphism $f \in {\rm Diff}_0(M)$ can be decomposed into a product of diffeomorphisms supported on a ball (see \cite{Banyaga97, Mann16} for example). 
The fragmentation norm $\|f\|_{\mathrm{Frag}}$ is defined as the minimal number of such factors needed to express $f$.

Since every diffeomorphism on a ball with compact support can be written as a product of 2 commutators \cite{Tsuboi08}, we have ${\rm cl}(f) \leq 2 \| f \|_{\rm Frag}$. 
Thus, $\| \cdot \|_{\rm Frag}$ is unbounded if ${\rm cl}$ is unbounded. 
Therefore, as in the case of orientable surfaces \cite{Bowden--Hensel--Webb22}, we obtain the following corollary of Theorem \ref{thm:main}.

\begin{corollary} \label{cor:frag}
For $g \geq 3$, the commutator length and the fragmentation norm are unbounded on $\mathrm{Diff}_0(N_g)$.
\end{corollary}

Corollary \ref{cor:frag}  contrasts with the fact that both the commutator length and the fragmentation norm are bounded for any closed manifold $M$ of dimension other than 2 or 4 \cite{BIP08, Tsuboi08, Tsuboi12}.

Let ${\rm Homeo}_0(M)$ denote the group of homeomorphisms which are isotopic to the identity. As in \cite{Bowden--Hensel--Webb22}, we can also prove the case of homeomorphism groups.

\begin{theorem} \label{thm:homeo}
  For $g\geq 3$, the space of nontrivial quasimorphisms  $\widetilde{QH}(\mathrm{Homeo}_0(N_g))$ is infinite-dimensional.
\end{theorem}

To prove Theorem~\ref{thm:homeo}, we use the automatic continuity of homogeneous quasimorphisms on $\mathrm{Diff}_0(N_g)$ by Kotschick (Theorem \ref{thm:conti}).
We can also prove Theorem \ref{thm:homeo} by using the topological fine curve graph $C^{\dagger}(N_g)_{\rm top}$ (see Section \ref{Topological fine curve graph}).

At the end of the introduction, we mention recent developments after our work. For (possibly nonorientable) surfaces $F$ with non-empty boundary such that $\chi(F)<0$, Bowden, Hensel, and Webb proved that $\widetilde{QH}(\mathrm{Homeo}_0(F,\partial F))$ is non-trivial \cite{Bowden--Hensel--Webb24}.
After their work, B\"{o}ke proved that the space $\widetilde{QH}(\mathrm{Homeo}_0(N_2))$ is non-trivial \cite{Boke}. 
The authors do not know whether the space $\widetilde{QH}(\mathrm{Homeo}_0(N_1))$ is trivial or not.

\section{Preliminaries}

Let $N=N_{g,n}$ (resp. $S=S_{g,n}$) be a connected finite type nonorientable (resp. orientable) surface of genus $g$ with $n$ punctures. 
We will abbreviate $N_{g,0}$ as $N_{g}$ 
and $S_{g,0}$ as $S_{g}$. 
If we are not concerned with whether a given surface is orientable or not, we write $F$ for the surface.

\subsection{Pseudo-Anosov mapping class}

Let ${\rm MCG}(F)$ be the mapping class group of a surface $F$, that is, the group of isotopy classes of (orientation preserving if $F$ is orientable) self-homeomorphisms on $F$.
For a homeomorphism $\varphi$ on $F$ (orientation preserving if $F$ is orientable),
$[\varphi]_F \in {\rm MCG}(F)$ denotes the mapping class of $\varphi$. 

A homeomorphism $\phi \colon F \to F$ is called \emph{pseudo-Anosov} if there exists a transverse pair of measured foliations on $F$ and a real number $\lambda > 1$ such that 
the foliations are preserved by $\phi$ and their transverse measures are multiplied by $1/\lambda$ and $\lambda$.
An element $f \in {\rm MCG}(F)$ is called \emph{pseudo-Anosov} if there is a pseudo-Anosov homeomorphism $\phi \colon F \to F$ such that $f=[\phi]_F$.

Let $S$ be the orientable double cover of a nonorientable surface $N$, and let 
$\pi \colon S \to N$ be the covering map.
The map $\pi$ induces the injective homomorphism $\iota \colon {\rm MCG}(N) \to {\rm MCG}(S)$ 
(see \cite{Birman--Chillingworth72,Goncalves--Guaschi--Maldonado18,Szepietowski10}).

A mapping class $f \in{\rm MCG}(N)$ is of finite order, is reducible, or is pseudo-Anosov if and only if its image under 
$\iota$ has the corresponding property in ${\rm MCG}(S)$ (see Wu~\cite{Wu87}).
Moreover, according to Thurston~\cite{Thurston88}, 
the Nielsen--Thurston classification theorem also holds for nonorientable surfaces (see \cite{Wu87} for the proof).

\begin{theorem}{$($\cite{Thurston88}, \cite[Theorem 2]{Wu87}$)$} 
A homeomorphism $\phi$ on $N$ is reduced and is not of finite order if and
only if $\phi$ is isotopic to a pseudo-Anosov homeomorphism.
\end{theorem}


A closed curve $c$ in $F$ is \textit{filling} if every curve isotopic to $c$ intersects every essential closed curve in $F$. Note that $c$ is filling if and only if $F-c$ is homeomorphic to a disjoint union of disks.

The following result by Kra \cite{Kra81} provides a criterion for constructing pseudo-Anosov mapping classes on orientable surfaces:

\begin{theorem}{$($\cite{Kra81}$)$} \label{thm:Kra}
Let $S$ be a compact orientable surface with negative Euler characteristics and $p \in S$  a point.
Let $\varphi \in {\rm Diff}_0(S)$ such that $\varphi(p)=p$ and $\{\varphi_t\}_{0\leq t\leq 1}$ an isotopy from ${\rm id_S}$ to $\varphi$. 
Then, the mapping class $[\varphi]_{S-{p}} \in {\rm MCG}(S-{p})$ is pseudo-Anosov if and only if the curve $\{ \varphi_t (p) \}_{0\leq t\leq 1}$ is filling.
\end{theorem}

Let $P=\{ p_1 ,\dots , p_n\} \subset S$ be a finite set of points. 
Given a diffeomorphism $\varphi \in {\rm Diff}_0(S)$ such that $\varphi(P)=P$ and an isotopy $\{\varphi_t\}_{0\leq t\leq 1}$ from ${\rm id_S}$ to $\varphi$,
we say that a system of closed curves $\mathcal{C}=\{c_1,\dots,c_m\}$ in $S$ is \textit{associated with $\varphi$} if 
it consists of the paths $\{ \varphi_t(p_i) \}_{0 \leq t \leq 1}$ for each $i=1,\dots,n$.
Two systems of closed curves $\mathcal{C}=\{c_1,\dots,c_m\}$ and
$\mathcal{C}'=\{c'_1,\dots,c'_m\}$ are \textit{equivalent} 
if $c_i$ is homotopic to $c'_i$ for $i = 1, \cdots ,m$. 
A system of curves $\mathcal{C}=\{c_1,\dots,c_m\}$ in $S$ is said to be \textit{filling} if 
$S-\cup_{i=1}^m c'_i$ is homeomorphic to a disjoint union of disks for 
every system $\mathcal{C'}=\{c'_1,\dots,c'_m\}$ equivalent to $\mathcal{C}$.

Ichihara and Motegi~\cite{Ichihara--Motegi05} extended Kra's theorem (Theorem \ref{thm:Kra}) to orientable surfaces with multiple punctures.

\begin{theorem}{$($\cite[Corollary 1.3]{Ichihara--Motegi05}$)$} \label{thm:IM}
Let $S$ be a compact orientable surface with negative Euler characteristics and $P=\{ p_1 ,\dots , p_n\} \subset S$ a finite set of points in $S$. 
Let $\varphi \in {\rm Diff}_0(S)$ such that $\varphi(P)=P$ and $\mathcal{C}=\{c_1,\dots,c_m\}$ a system of curves associated with $\varphi$. If $\mathcal{C}$ is filling and satisfies
\begin{itemize}
    \item every $c_i$ is primitive (i.e., not freely homotopic to a $k$-fold closed curve $c_i^k$ of $c_i$ with $k \geq 2$), 
    \item $c_i$ and $c_j$ are not freely homotopic for $i \neq j$, and 
    \item $c_i$ cannot be homotoped into $\partial S$, 
\end{itemize}
then the mapping class $[\varphi]_{S-P} \in {\rm MCG}(S-P)$ is pseudo-Anosov. 
\end{theorem}

We will use Theorem \ref{thm:IM} to construct a pseudo-Anosov mapping class $[\varphi]_{N-p} \in {\rm MCG}(N-p)$ so that $\varphi \in \mathrm{Diff}_0(N)$ (Lemma \ref{lem:pA}). 

\subsection{Geometric group theory}
We briefly review the basics of geometric group theory relevant to this paper (for more details, see \cite{Drutu-Kapovich,Loh} for example).

\begin{definition}
Let $(X,d_X)$ and $(Y,d_Y)$ be metric spaces. 
\begin{enumerate}
    \item  For $K\geq 1$ and $L\geq 0$, a map $f \colon X \to Y$ is a \emph{$(K,L)$-quasi-isometric embedding} if
    \[ K^{-1} d_X(x_1,x_2) -L \leq d(f(x_1),f(x_2)) \leq K d_X(x_1,x_2) +L \]
    for every $x_1,x_2 \in X$.
    \item  For $C\geq0$, a map $f \colon X \to Y$ is \emph{$C$-dense} if for every $y \in Y$, there exists $x \in X$ such that
    \[ d_Y(d(x),y) \leq C. \]
    \item A map $f \colon X \to Y$ is a \emph{quasi-isometry} if $f$ is a $(K,L)$-quasi-isometric embedding and \emph{$C$-dense} for some $K$, $L$ and $C$. We shall call these constants $K$, $L$ and $C$ the \emph{quasi-isometric constants} of $f$.
    \item The metric spaces $X$ and $Y$ are \emph{quasi-isometric} if there exists a quasi-isometry from $X$ to $Y$, and we write $X \qisim Y$.
\end{enumerate}

\end{definition}

Note that a $(1,0)$-quasi-isometric embedding is nothing but an isometric embedding.
We remark that $\qisim$ is an equivalence relation.
In particular, 
if $f\colon X \to Y$ and $g\colon Y \to Z$ are quasi-isometries, then $g \circ f$ is also a quasi-isometry, and the quasi-isometric constant of $g \circ f$ depends only on those of $f$ and $g$.

The Gromov hyperbolicity of metric spaces is defined as follows:
\begin{definition}
Let $(X,d)$ be a metric space. For points $x$, $y$, $w$ of $(X,d)$, the \emph{Gromov product} is defined to be
\begin{equation*}
\langle x,y\rangle_{w}\coloneqq\frac{1}{2}(d(w,x)+d(w,y)-d(x,y)).
\end{equation*}
For $\delta\geq 0$, a metric space $X$ is \emph{$\delta$-hyperbolic} if for all $w,x,y,z\in X$ we have
\begin{equation*}
\langle x,z\rangle_{w}\geq\mathrm{min}\{\langle x,y\rangle_{w},\langle y,z\rangle_{w}\}-\delta.
\end{equation*}
A metric space is called \emph{Gromov hyperbolic} if it is $\delta$-hyperbolic for some $\delta$.
We call the constant $\delta$ the \emph{hyperbolicity constant} of the space.

\end{definition}
It is well-known that the Gromov hyperbolicity is preserved under quasi-isometries.
Moreover, if $f \colon X \to Y$ is a quasi-isometry between Gromov hyperbolic spaces, then the hyperbolicity constant of $X$ depends only on that of $Y$ and the quasi-isometry constants of $f$.

\subsection{Curve graph}

A graph $\Gamma=(V,E)$ consists of a set of \emph{vertices} $V$ and a set of \emph{edges} $E$, where each edge is an unordered pair of vertices. 
A graph $\Gamma$ is \emph{connected} 
if any two vertices can be joined by a path of edges (i.e., for every $v,w \in V$, there exist edges $\{v_0,v_1\}, \{v_1,v_2\}, \dots, \{ v_{n-1}, v_n \}$ such that $v_0=v$ and $v_n=w$). We can regard a connected graph as a metric space, and furthermore as a geodesic space, by assigning length 1 to each edge.

A simple closed curve $c$ on $F$ is {\it essential} if either of the following holds:
\begin{itemize}
    \item $c$ is nonseparating.
    \item $c$ is separating, does not bound a disk, a disk with one marked point, or a M\"obius band, and is not isotopic to a boundary component of $F$.
\end{itemize}
The {\it curve graph} $\mathcal{C}(F)$ of $F$ 
is the graph whose vertices are isotopy classes of essential simple closed curves on $N$. Two vertices are connected by an edge if the corresponding curves can be realized disjointly.
The  {\it nonseparating curve graph} $\mathcal{NC}(F)$ of $F$ is the full subgraph of 
$\mathcal{C}(F)$ consisting of vertices represented by nonseparating curves.

The {\it geometric intersection number} $i(c_{1},c_{2})$ between simple closed curves $c_{1}$ and $c_{2}$ on $F$ is defined to be the number of intersection points between $c_{1}$ and $c_{2}$. For a curve $c$ on $F$, we denote by $[c]_{F}$ the homotopy class of $c$ in $F$. Then we also denote by $i([c_{1}]_{F},[c_{2}]_{F})$ the {\it geometric intersection number} between the isotopy classes $c_{1}$ and $c_{2}$ of simple closed curves on $F$, that is, the minimal number of intersection points between a representative curve in the class of $c_{1}$ and a representative curve in the class of $c_{2}$. Two simple closed curves $c_{1}$ and $c_{2}$ in $F$ are in {\it minimal position} if the number of intersections of $c_{1}$ and $c_{2}$ is minimal in the isotopy classes of $c_{1}$ and $c_{2}$. Note that two essential simple closed curves are in minimal position in $F$ if and only if they do not bound a bigon on $F$ (see~\cite[Proposition 2.1]{Stukow17} for nonorientable surfaces). 

Slightly abusing the notation, we consider vertices of curve graphs as the essential simple closed curves on $F$ which are in minimal position.
Furthermore, so long as it does not cause confusion, 
we may refer to an essential simple closed curve on $F$ as a curve.
We define the distances $d_{\mathcal{C}}(\cdot, \cdot)$ on $\mathcal{C}(N)$ by the minimal lengths of edge-paths connecting two vertices. Thus, we consider $\mathcal{C}(N)$, as a geodesic space.

Bowden, Hensel, and Webb~\cite{Bowden--Hensel--Webb22} define a new curve graph $\mathcal{C}^{\dagger}(S)$ called the {\it fine curve graph} for a closed orientable surface $S$ and prove that it is Gromov hyperbolic with a hyperbolicity constant that does not depend on the topological type of $S$ by using the uniform hyperbolicity of $\mathcal{NC}(S)$.

We define the {\it fine curve graph} of a nonorientable surface $N$ in accordance with Bowden, Hensel, and Webb~\cite{Bowden--Hensel--Webb22}.

\begin{definition}
Let $N=N_{g,n}$ be a nonorientable surface of genus $g\geq 3$ with $n\geq 0$ punctures. A {\it fine curve graph} $\mathcal{C}^{\dagger}(N)$ of $N$ is a graph whose vertices are the essential simple closed curves on $N$, and two vertices form an edge if the corresponding curves are disjoint.
\end{definition}

We remark that for $g=1,2$, we need to modify the definition of the fine curve graph (See Sention \ref{Low_genus_cases}). 

In this paper, we will prove that the fine curve graphs $\mathcal{C}^{\dagger}(N)$ of closed nonorientable surfaces are uniformly hyperbolic by using Bowden, Hensel, and Webb's~\cite{Bowden--Hensel--Webb22} technique:

\begin{theorem}\label{unif_hyp_fine_curve_graph}
For every closed nonorientable surface $N=N_{g}$ of genus $g\geq 3$, $\mathcal{C}^{\dagger}(N)$ is Gromov-hyperbolic. Moreover, its hyperbolicity constant can be taken to be independent of $g$.
\end{theorem}

Such a phenomenon, where the hyperbolicity constant can be taken uniformly, is known as \emph{uniform hyperbolicity}.
The proof of Theorem~\ref{unif_hyp_fine_curve_graph} will be given in Section~\ref{Uniform_hyperbolicity_for_the_fine_curve_graphs_of_nonorientable surfaces}.
The case of $g=2$ will be discussed in Section \ref{Low_genus_cases}.

\subsection{Quasimorphism} \label{subsec:qm}

Let $G$ be a group. A function $\phi \colon G \to \mathbb{R}$ is called a \emph{quasimorphism} if its \emph{defect}
\[ D(\phi) = \sup_{g,h \in G} |\phi(gh)-\phi(g)-\phi(h)| \]
is finite.
A quasimorphism $\phi$ is \emph{homogeneous} if $\phi(g^n)=n\phi(g)$ for every $g \in G$ and $n \in \mathbb{Z}$.
Let $QH(G)$ denote the $\mathbb{R}$-vector space of homogeneous quasimorphisms on $G$.
Naturally, the space $H^1(G)=H^1(G;\mathbb{R})$ of homomorphisms from $G$ to $\mathbb{R}$ is a linear subspace of $QH(G)$.
The quotient space $QH(G)/H^1(G)$ is called \emph{the space of nontrivial quasimorphisms} and is denoted by $\widetilde{QH}(G)$.
It is well-known that $\widetilde{QH}(G)$ is isomorphic to the kernel of the comparison map $H_b^2(G) \to H^2(G)$. Therefore, the study of second bounded cohomology can be attributed to the study of quasimorphisms. For more information about quasimorphisms and bounded cohomology, see \cite{Calegari09} for example.

Bestvina and Fujiwara~\cite{Bestvina--Fujiwara07} developed a method for constructing quasimorphisms using group actions on hyperbolic spaces. 
Assume that a group $G$ acts on a graph $\Gamma$ isometrically.
A \textit{translation length} $|g|_{\Gamma}$ of $g \in G$ is defined by
\[ |g|_{\Gamma} = \lim_{n\to \infty} \frac{d(x,g^n x)}{n} \]
for some (any) $x \in \Gamma$.
We say that $g$ is \textit{hyperbolic} (with respect to the action on $\Gamma$) if $|g|_{\Gamma}>0$.
For a subset $A$ of $\mathbb{R}$, a $(K,L)$-quasi-isometric embedding $\gamma \colon A \to \Gamma$ (and its image $\gamma(A)$) is called a \textit{$(K,L)$-quasi-geodesic}.
A \textit{$(K,L)$-quasi-axis} means a $(K,L)$-quasi-geodesic which is invariant under the action of $g \in G$.


The following proposition provides a sufficient condition for the existence of infinitely many nontrivial quasimorphisms:

\begin{proposition}[\cite{Bestvina--Fujiwara07},{\cite[Corollary 2.7.]{Bowden--Hensel--Webb22}}] \label{prop:BF}
  Suppose that a group $G$ acts on a $\delta$-hyperbolic space $X$.
  Let $g_1,g_2 \in G$ be two hyperbolic elements with $(K,L)$-quasi-axis $A_1,A_2$.
  Suppose that for any $x\in A_1$ and any $h \in G$ we have
  \[ d(hx,A_2) > B(K,L,\delta). \]
  Then, the space of nontrivial quasimorphisms $\widetilde{QH}(G)$ is infinite-dimensional.
\end{proposition}
Here, $B(K,L,\delta)$ is a constant which comes from the Morse lemma for quasi-geodesics and depends only on $K$, $L$, and $\delta$ (see \cite{Bestvina--Fujiwara07,Bowden--Hensel--Webb22} for the definition).

\section{Uniform hyperbolicity for the fine curve graphs of nonorientable surfaces}\label{Uniform_hyperbolicity_for_the_fine_curve_graphs_of_nonorientable surfaces}

In this section, we prove Theorem~\ref{unif_hyp_fine_curve_graph}, that is, the fine curve graphs $\mathcal{C}^{\dagger}(N)$ of closed nonorientable surfaces are uniformly hyperbolic. 

\subsection{Notation}

Let $N=N_{g,n}$ be a nonorientable surface of genus $g\geq 3$ with $n\geq 0$ punctures.

\begin{definition}
The {\it surviving curve graph} $\mathcal{C}^{s}(N)$ is a full subgraph of the original curve graph $\mathcal{C}(N)$ whose vertices correspond to the isotopy classes of curves on $N$ which are essential even after filling in the punctures.
\end{definition}

Note that every nonseparating curve is surviving, meaning that $\mathcal{NC}(N)$ is a full subgraph of $\mathcal{C}^{s}(N)$.


A curve $c$ on a nonorientable surface $N$ is called {\it one-sided} if its regular neighborhood is a M\"obius band, and {\it two-sided} if the regular neighborhood is an annulus.

We define two variants of the graph $\mathcal{C}(N)$ as follows:
\begin{itemize}
    \item The {\it two-sided curve graph} $\mathcal{C}_{\mathrm{two}}(N)$ of $N$ is the full subgraph of $\mathcal{C}(N)$ consisting of vertices corresponding to two-sided essential curves.
    \item The {\it extended curve graph} $\mathcal{C}^{\pm}(N)$ of $N$ is 
    the curve graph whose vertices are 
    the isotopy classes not only of essential simple closed curves but also of simple closed curves bounding a M\"obius band.
    Two vertices are connected by an edge if the corresponding curves can be realized disjointly.
\end{itemize}
The above notation (two-sided and extended) will similarly apply to
$\mathcal{NC}(N)$, $\mathcal{C}^{s}(N)$, and $\mathcal{C}^{\dagger}(N)$. 
Moreover, we may use multiple notations simultaneously; 
for instance, $\mathcal{C}^{\pm\dagger}_{\mathrm{two}}(N)$ refers to the extended two-sided fine curve graph of $N$ whose vertices consist of both the two-sided curves and curves bounding a M\"obius band.

We define the distances $d_{\mathcal{C}^{\pm}}(\cdot,\cdot)$ on $\mathcal{C}^{\pm}(N)$, for example, in terms of the minimal lengths of edge paths connecting the two vertices (the same symbols are used for the other graphs). Thus, we consider them to be geodesic spaces.

\subsection{Quasi-isometries between curve graphs}

\begin{proposition}
Let $N=N_{g,n}$ be a nonorientable surface of genus $g\geq 3$ with $n\geq 0$ punctures. Then, $\mathcal{C}^{s}(N)$, $\mathcal{C}^{\pm s}(N)$, and $\mathcal{C}^{\pm s}_{\mathrm{two}}(N)$ are path-connected.
\end{proposition}

\begin{proof}
First, we show the path-connectedness of $\mathcal{C}^{s}(N)$ (the same argument goes for $\mathcal{C}^{\pm s}(N)$). Take any pair of vertices $a, b\in\mathcal{C}^{s}(N)$. If both $a$ and $b$ are nonseparating, then by the path-connectedness of $\mathcal{NC}(N)$, 
they can be joined by a path
in $\mathcal{NC}(N)\subset\mathcal{C}^{s}(N)$. We may assume that at least one of $a$ and $b$ is separating. 
Since $a$ and $b$ are surviving, the separating one bounds a subsurface with a positive genus 
(if a subsurface is nonorientable, the genus is at least two if we consider $\mathcal{C}^{s}(N)$, and at least one if we consider $\mathcal{C}^{\pm s}(N)$). We can take nonseparating curves $a'$ and/or $b'$ that are disjoint from $a$ and $b$, respectively. Since $a'$ and $b'$ are nonseparating, they are connected by a path in $\mathcal{NC}(N)\subset\mathcal{C}^{s}(N)$. Hence, there is a path connecting $a$ and $b$ in $\mathcal{C}^{s}(N)$.

Second, we show that $\mathcal{C}^{\pm s}_{\mathrm{two}}(N)$ is path-connected. Take any pair of vertices $a, b\in\mathcal{C}^{\pm s}_{\mathrm{two}}(N)$. There exists a path $\mathcal{P}^{\pm s}(a,b)$ connecting $a$ and $b$ in $\mathcal{C}^{\pm s}(N)$. Let $d$ be any one-sided curve in $\mathcal{P}^{\pm s}(a,b)$ and $c'$ and $c''$ be the vertices adjacent to $d$ in $\mathcal{P}^{\pm s}(a,b)$. The boundary $c$ of the regular neighborhood of $d$ is a two-sided curve bounding a M\"obius band. Note that $i(d,c')=i(d,c'')=0$, which implies $i(c,c')=i(c,c'')=0$. Hence, we can replace $d$ by $c$ in $\mathcal{P}^{\pm s}(a,b)$. By repeating this process for all one-sided curves in $\mathcal{P}^{\pm s}(a,b)$, we obtain a path connecting $a$ and $b$ in $\mathcal{C}^{\pm s}_{\mathrm{two}}(N)$.
\end{proof}

\begin{proposition}\label{nonsep_curve_graph_and_surviving_curve_graph_are_quasi-isometric}
Let $N=N_{g,n}$ be a nonorientable surface of genus $g\geq 3$ with $n\geq 0$ punctures. Then, $\mathcal{C}^{s}(N)$, $\mathcal{C}^{\pm s}(N)$, and $\mathcal{C}^{\pm s}_{\mathrm{two}}(N)$ are quasi-isometric:
\begin{equation*}
 \mathcal{NC}(N) \qisim \mathcal{C}^{s}(N) \qisim \mathcal{C}^{\pm s}(N) \qisim \mathcal{C}^{\pm s}_{\mathrm{two}}(N).
\end{equation*}
Moreover, we can take quasi-isometries between them so that their quasi-isometric constants do not depend on $g$ and $n$.
\end{proposition}

Proposition \ref{nonsep_curve_graph_and_surviving_curve_graph_are_quasi-isometric} follows from the following Lemmas \ref{f1}, \ref{f2} and \ref{f3}.

\begin{lemma} \label{f1}
For $N=N_{g,n}$ with $g\geq 3$ and $n\geq 0$, the natural inclusion map $f_{1}\colon\mathcal{NC}(N)\hookrightarrow\mathcal{C}^{s}(N)$ is an isometric embedding 
and is 1-dense.
\end{lemma}

\begin{proof} First, we prove that $f_{1}$ is an isometric embedding. Since $\mathcal{NC}(N)\subset\mathcal{C}^{s}(N)$, for any $a,b\in\mathcal{NC}(N)$ we have $d_{\mathcal{C}^{s}}(a,b)\leq d_{\mathcal{NC}}(a,b)$. To show $d_{\mathcal{C}^{s}}(a,b)\geq d_{\mathcal{NC}}(a,b)$ for any $a,b\in\mathcal{NC}(N)$, let $\mathcal{G}^{s}(a,b)$ be a geodesic connecting $a$ and $b$ in $\mathcal{C}^{s}(N)$. Take any separating curve $d$ in $\mathcal{G}^{s}(a,b)$. Since $d$ is surviving, each connected component of $N\setminus d$ has a positive genus. 
Let $F$ and $F'$ denote these components.
Note that if the connected component is a nonorientable surface, 
its genus is at least 2.
Let $c'$ and $c''$ be the vertices adjacent to $d$ in $\mathcal{G}^{s}(a,b)$. Then $c'$ and $c''$ intersect since $\mathcal{G}^{s}(a,b)$ is a geodesic. 
Hence, they lie in the same component $F$.
We can take a nonseparating curve $c$ in $F'$, and the path where $d$ is replaced by $c$ in $\mathcal{G}^{s}(a,b)$ is also a geodesic. By repeating this process for all separating curves in $\mathcal{G}^{s}(a,b)$, we obtain a geodesic consisting of nonseparating curves. Hence, $d_{\mathcal{NC}}(a,b)\leq d_{\mathcal{C}^{s}}(a,b)$. 

Next, we prove that $f_{1}$ is 1-dense. If $a\in\mathcal{C}^{s}(N)$ is a nonseparating curve,  we take $a\in\mathcal{NC}(N)$, and we have $d_{\mathcal{C}^{s}}(a,f_{1}(a))=d_{\mathcal{C}^{s}}(a,a)=0$. If $a\in\mathcal{C}^{s}(N)$ is a separating curve, then each component of $N\setminus a$ has a positive genus since $a$ is surviving. We take a nonseparating curve $b$ in the component; then, we have $d_{\mathcal{C}^{s}}(a,f_{1}(b))=d_{\mathcal{C}^{s}}(a,b)=1$.
\end{proof}

\begin{lemma} \label{f2}
For $N=N_{g,n}$ with $g\geq 3$ and $n\geq 0$, the natural inclusion map $f_{2}\colon\mathcal{C}^{s}(N)\hookrightarrow\mathcal{C}^{\pm s}(N)$ is an isometric embedding 
and is 1-dense.
\end{lemma}

\begin{proof} 
 First, we prove that $f_{2}$ is an isometric embedding. Since $\mathcal{C}^{s}(N)\subset\mathcal{C}^{\pm s}(N)$, for any $a,b\in\mathcal{C}^{s}(N)$ we have $d_{\mathcal{C}^{\pm s}}(a,b)\leq d_{\mathcal{C}^{s}}(a,b)$. To show $d_{\mathcal{C}^{\pm s}}(a,b)\geq d_{\mathcal{C}^{s}}(a,b)$ for any $a,b\in\mathcal{C}^{s}(N)$, let $\mathcal{G}^{\pm s}(a,b)$ be a geodesic connecting $a$ and $b$ in $\mathcal{C}^{\pm s}(N)$. Take any curve $d$ bounding a M\"obius band in $\mathcal{G}^{\pm s}(a,b)$. We can replace $d$ by the core curve $c$ of the M\"obius band in $\mathcal{G}^{\pm s}(a,b)$. The resulting path is a geodesic that does not have curves bounding a M\"obius band. Therefore, $d_{\mathcal{C}^{s}}(a,b)\leq d_{\mathcal{C}^{\pm s}}(a,b)$. 
 
 Next, we prove that $f_{2}$ is 1-dense. If $a\in\mathcal{C}^{\pm s}(N)$ is a curve bounding a M\"obius band, we take the core curve $b\in\mathcal{C}^{s}(N)$ of the M\"obius band, and we have $d_{\mathcal{C}^{\pm s}}(a,f_{2}(b))=d_{\mathcal{C}^{\pm s}}(a,b)=1$. Otherwise, we have $d_{\mathcal{C}^{\pm s}}(a,f_{2}(a))=d_{\mathcal{C}^{\pm s}}(a,a)=0$.
\end{proof}

\begin{lemma} \label{f3}
For $N=N_{g,n}$ with $g\geq 3$ and $n\geq 0$, the natural inclusion map $f_{3}\colon\mathcal{C}^{\pm s}_{\mathrm{two}}(N)\hookrightarrow\mathcal{C}^{\pm s}(N)$ is an isometric embedding 
and is 1-dense.
\end{lemma}

\begin{proof} First, we prove that $f_{3}$ is an isometric embedding. Since $\mathcal{C}^{\pm s}_{\mathrm{two}}(N)\subset\mathcal{C}^{\pm s}(N)$, for any $a,b\in\mathcal{C}^{\pm s}_{\mathrm{two}}(N)$ we have $d_{\mathcal{C}^{\pm s}}(a,b)\leq d_{\mathcal{C}^{\pm s}_{\mathrm{two}}}(a,b)$. To show $d_{\mathcal{C}^{\pm s}}(a,b)\geq d_{\mathcal{C}^{\pm s}_{\mathrm{two}}}(a,b)$ for any $a,b\in\mathcal{C}^{\pm s}_{\mathrm{two}}(N)$, let $\mathcal{G}^{\pm s}(a,b)$ be a geodesic connecting $a$ and $b$ in $\mathcal{C}^{\pm s}(N)$. Take any one-sided curve $d$ in $\mathcal{G}^{\pm s}(a,b)$. We can replace $d$ by the boundary $c$ of the regular neighborhood of $d$ in $\mathcal{G}^{\pm s}(a,b)$. The resulting path is a geodesic consisting of two-sided curves. Therefore, $d_{\mathcal{C}^{\pm s}_{\mathrm{two}}}(a,b)\leq d_{\mathcal{C}^{\pm s}}(a,b)$. 

Next, we prove that $f_{3}$ is 1-dense. If $a\in\mathcal{C}^{\pm s}(N)$ is two-sided, we take $a\in\mathcal{C}^{\pm s}_{\mathrm{two}}(N)$, and $d_{\mathcal{C}^{\pm s}}(a,f_{3}(a))=d_{\mathcal{C}^{\pm s}}(a,a)=0$. If $a\in\mathcal{C}^{\pm s}(N)$ is one-sided, we take the boundary $b$ of the regular neighborhood of $a$. Then $b\in\mathcal{C}^{\pm s}_{\mathrm{two}}(N)$, and $d_{\mathcal{C}^{\pm s}}(a,f_{3}(b))=d_{\mathcal{C}^{\pm s}}(a,b)=1$.
\end{proof}

It is known that the nonseparating curve graphs $\mathcal{NC}(N)$ of nonorientable surfaces $N$ are uniformly hyperbolic (see~\cite{Kuno21}):

\begin{theorem}{\rm(}\cite{Kuno21}{\rm)}\label{unif_hyp_of_nonseparating_curve_graph}
For every nonorientable surface $N=N_{g,n}$ of genus $g\geq 3$ with $n\geq 0$ punctures, the nonseparating curve graph $\mathcal{NC}(N)$ is connected and Gromov-hyperbolic. Moreover, its hyperbolicity constant can be taken to be independent of $g$ and $n$.
\end{theorem}

By combining Proposition~\ref{nonsep_curve_graph_and_surviving_curve_graph_are_quasi-isometric} and Theorem~\ref{unif_hyp_of_nonseparating_curve_graph}, we obtain the following corollary:

\begin{corollary}\label{extended_two-sided_cur_graph_is_uniformly_hyperbolic}
For every nonorientable surface $N=N_{g,n}$ of genus $g\geq 3$ with $n\geq0$ punctures, the extended two-sided surviving curve graph $\mathcal{C}^{\pm s}_{\mathrm{two}}(N)$ is  Gromov-hyperbolic. 
Moreover, its hyperbolicity constant can be taken to be independent of $g$ and $n$.
\end{corollary}

\subsection{Quasi-isometries between fine curve graphs}

\begin{proposition}\label{fine_curve_graphs_are_path_connected}
Let $N=N_{g,n}$ be a nonorientable surface of genus $g\geq 3$ with $n\geq 0$ punctures. Then, $\mathcal{C}^{\dagger}(N)$, $\mathcal{C}^{\pm \dagger}(N)$, and $\mathcal{C}^{\pm \dagger}_{\mathrm{two}}(N)$ are path-connected.
\end{proposition}

We will require the following lemma to prove Proposition~\ref{fine_curve_graphs_are_path_connected}:

\begin{lemma}\label{extended_two-sided_curve_graph_is_path-conn}
Let $N=N_{g,n}$ be a nonorientable surface of genus $g\geq 3$ with $n\geq 0$ punctures. Then, the extended two-sided curve graph $\mathcal{C}^{\pm}_{\mathrm{two}}(N)$ is path-connected.
\end{lemma}

\begin{proof}[Proof of Lemma~\ref{extended_two-sided_curve_graph_is_path-conn}]
Note that the extended curve graph $\mathcal{C}^{\pm}(N)$ is path-connected because the ordinary curve graph $\mathcal{C}(N)$ is path-connected. Let $a,b\in\mathcal{C}^{\pm}_{\mathrm{two}}(N)$ be any pair of vertices which cannot be realized disjointly. We connect $a$ and $b$ by a geodesic $\mathcal{G}^{\pm}(a,b)$ in $\mathcal{C}^{\pm}(N)$. We choose any isotopy class of a one-sided curve $d$ in $\mathcal{G}^{\pm}(a,b)$. Let $c'$ and $c''$ be the vertices adjacent to $d$ in $\mathcal{G}^{\pm}(a,b)$. The boundary $e$ of the regular neighborhood of $d$ is a two-sided curve bounding a M\"obius band, so $e\in\mathcal{C}^{\pm}_{\mathrm{two}}(N)$. Moreover, $i(e,c')=i(e,c'')=0$. Hence, the path on which we replace $d$ with $e$ is also a geodesic in $\mathcal{C}^{\pm}(N)$. By repeating this process for all one-sided curves in $\mathcal{G}^{\pm}(a,b)$, we obtain a path in $\mathcal{C}^{\pm}_{\mathrm{two}}(N)$ connecting $a$ and $b$. Therefore, the extended two-sided curve graph $\mathcal{C}^{\pm}_{\mathrm{two}}(N)$ is path-connected.
\end{proof}

\begin{remark}
Lemma~\ref{extended_two-sided_curve_graph_is_path-conn} holds whenever the ordinary curve graph $\mathcal{C}(N)$ is path-connected; that is, if $g=1,2$, then $g+n\geq 5$ and $g\geq 3$. This condition implies that the {\it two-sided complexity}
\[ \xi^{\mathrm{two}}(N)=
 \begin{cases}
 \frac{3}{2}(g-1)+n-2 & (\mathrm{if}~g~\mathrm{is~odd}),\\
 \frac{3}{2}g+n-3 & (\mathrm{if}~g~\mathrm{is~even}),\\
 \end{cases} \]
of $N$, that is, the cardinality of a maximal set of mutually disjoint two-sided simple closed curves in $N_{g,n}$, is at least one.
\end{remark}

\begin{proof}[Proof of Proposition~\ref{fine_curve_graphs_are_path_connected}]
First, we show the path-connectedness for $\mathcal{C}^{\pm \dagger}_{\mathrm{two}}(N)$. Let $a, b\in\mathcal{C}^{\pm \dagger}_{\mathrm{two}}(N)$ be any pair of vertices which intersect. If $a$ and $b$ are not isotopic, we can connect them by representatives which are mutually in minimal position of a path in $\mathcal{C}^{\pm}_{\mathrm{two}}(N)$, by Lemma~\ref{extended_two-sided_curve_graph_is_path-conn}. We may assume that $a$ and $b$ are isotopic. Since $a$ and $b$ are two-sided, $a$ can be isotoped to a curve $c$ so that $i(a,c)=0$ and $i(c,b)=0$. Hence, we can connect $a$ and $b$ by a path in $\mathcal{C}^{\pm \dagger}_{\mathrm{two}}(N)$.

Next, we show that $\mathcal{C}^{\pm \dagger}(N)$ is path-connected. Let $a, b\in\mathcal{C}^{\pm \dagger}(N)$ be any pair of vertices which intersect. It is enough to show that $a$ and $b$ are isotopic and one-sided. In this case, we take a boundary $a'$ of the regular neighborhood of $a$ so that $i(a,a')=0$. The curve $a$ can be isotoped to $b$ by some isotopy $\iota$. Then, $a'$ is also isotoped by $\iota$ to a curve which is a boundary of a regular neighborhood of $b$, and denote the curve by $b'$. Then, it follows that $a'$ and $b'$ are isotopic and two-sided. By the same argument presented above, we can connect $a$ and $b$ by a path in $\mathcal{C}^{\pm \dagger}(N)$.

Finally, we show that $\mathcal{C}^{\dagger}(N)$ is path-connected. Let $a, b\in\mathcal{C}^{\dagger}(N)$ be any pair of vertices which intersect. Let $\mathcal{G}^{\pm\dagger}(a,b)$ denote a geodesic connecting $a$ and $b$ in $\mathcal{C}^{\pm \dagger}(N)$. Let $d$ be any curve contained in $\mathcal{G}^{\pm\dagger}(a,b)$ which bounds a M\"obius band $M$ in the nonorientable surface $N$, and let $c'$ and $c''$ be the vertices adjacent to $d$ in $\mathcal{G}^{\pm\dagger}(a,b)$. Since $\mathcal{G}^{\pm\dagger}(a,b)$ is a geodesic, $c'$ and $c''$ must intersect. Hence, both $c'$ and $c''$ lie on either $M$ or $N\setminus M$. Now, the genus of $N$ is at least three, so we can take an essential curve $c$ from the component that does not contain $c'$ and $c''$. The resulting path on which we replace $d$ with $c$ in $\mathcal{G}^{\pm\dagger}(a,b)$ is also a geodesic in $\mathcal{C}^{\pm \dagger}(N)$. By repeating this process for all curves in $\mathcal{G}^{\pm\dagger}(a,b)$ bounding a M\"obius band, we obtain a geodesic connecting $a$ and $b$ in $\mathcal{C}^{\pm \dagger}(N)$ that consists of essential curves in $N$. This geodesic is a path in $\mathcal{C}^{\dagger}(N)$ connecting $a$ and $b$, so $\mathcal{C}^{\dagger}(N)$ is path-connected.
\end{proof}

\begin{proposition}\label{fine_curve_graphs_are_quasi-isometric}
Let $N=N_{g,n}$ be a nonorientable surface of genus $g\geq 3$ with $n\geq 0$ punctures. Then, $\mathcal{C}^{\dagger}(N)$, $\mathcal{C}^{\pm \dagger}(N)$, and $\mathcal{C}^{\pm \dagger}_{\mathrm{two}}(N)$ are quasi-isometric:
\begin{equation*}
 \mathcal{C}^{\dagger}(N) \qisim \mathcal{C}^{\pm \dagger}(N) \qisim \mathcal{C}^{\pm\dagger}_{\mathrm{two}}(N).
\end{equation*}    
Moreover, we can take quasi-isometries between them so that their quasi-isometric constants do not depend on $g$ and $n$.
\end{proposition}

Proposition \ref{nonsep_curve_graph_and_surviving_curve_graph_are_quasi-isometric} follows from the following Lemmas \ref{f4} and \ref{f5}.

\begin{lemma} \label{f4}
For $N=N_{g,n}$ with $g\geq 3$ and $n\geq 0$, the natural inclusion map  $f_{4}\colon\mathcal{C}^{\pm\dagger}_{\mathrm{two}}(N)\hookrightarrow\mathcal{C}^{\pm\dagger}(N)$ is  an isometric embedding and is 1-dense.
\end{lemma}

\begin{proof}
First, we prove that $f_{4}$ is an isometric embedding. Let $a,b\in\mathcal{C}^{\pm\dagger}_{\mathrm{two}}(N)$ be any pair of vertices.
Since $\mathcal{C}^{\pm\dagger}_{\mathrm{two}}(N)\subset\mathcal{C}^{\pm\dagger}(N)$, for any $a,b\in\mathcal{C}^{\pm\dagger}_{\mathrm{two}}(N)$, we have
$d_{\mathcal{C}^{\pm\dagger}}(a,b)\leq d_{\mathcal{C}^{\pm\dagger}_{\mathrm{two}}}(a,b)$.
To show $d_{\mathcal{C}^{\pm\dagger}}(a,b)\geq d_{\mathcal{C}^{\pm\dagger}_{\mathrm{two}}}(a,b)$ for any $a,b\in\mathcal{C}^{\pm\dagger}_{\mathrm{two}}(N)$, let $\mathcal{G}^{\pm\dagger}(a,b)$ be a geodesic connecting $a$ and $b$ in $\mathcal{C}^{\pm\dagger}(N)$.
Take any one-sided curve $d$ in $\mathcal{G}^{\pm\dagger}(a,b)$. We can replace $d$ with the boundary $c$ of the regular neighborhood of $d$ which does not intersect with $d$ in $\mathcal{G}^{\pm\dagger}(a,b)$. After repeating this process for all one-sided curves in $\mathcal{G}^{\pm\dagger}(a,b)$, we obtain a geodesic in $\mathcal{C}^{\pm\dagger}(N)$ that consists of two-sided curves.
Therefore, $d_{\mathcal{C}^{\pm\dagger}_{\mathrm{two}}}(a,b)\leq d_{\mathcal{C}^{\pm\dagger}}(a,b)$. 

Next, we prove that $f_{4}$ is 1-dense. It is enough to consider the one-sided curve $a\in\mathcal{C}^{\pm\dagger}(N)$. In this case, we take the boundary $b$ of the regular neighborhood of $a$ which does not intersect with $a$.
Then, $b\in\mathcal{C}^{\pm\dagger}_{\mathrm{two}}(N)$, and $d_{\mathcal{C}^{\pm\dagger}}(a,f_{4}(b))=d_{\mathcal{C}^{\pm\dagger}}(a,b)=1$.
\end{proof}

\begin{lemma} \label{f5}
For $N=N_{g,n}$ with $g\geq 3$ and $n\geq 0$, the natural inclusion map $f_{5}\colon\mathcal{C}^{\dagger}(N)\hookrightarrow\mathcal{C}^{\pm\dagger}(N)$ is  an isometric embedding and is 1-dense.
\end{lemma}

\begin{proof}
First, we prove that $f_{5}$ is an isometric embedding. Let $a,b\in\mathcal{C}^{\dagger}(N)$ be any pair of vertices.
Since $\mathcal{C}^{\dagger}(N)\subset\mathcal{C}^{\pm\dagger}(N)$, for any $a,b\in\mathcal{C}^{\dagger}(N)$, we have $d_{\mathcal{C}^{\pm\dagger}}(a,b)\leq d_{\mathcal{C}^{\dagger}}(a,b)$.
To show $d_{\mathcal{C}^{\pm\dagger}}(a,b)\geq d_{\mathcal{C}^{\dagger}}(a,b)$ for any $a,b\in\mathcal{C}^{\dagger}(N)$, let $\mathcal{G}^{\pm\dagger}(a,b)$ be a geodesic connecting $a$ and $b$ in $\mathcal{C}^{\pm\dagger}(N)$. Take any curve $d$ bounding a M\"obius band $M$ in $\mathcal{G}^{\pm\dagger}(a,b)$. Let $c'$ and $c''$ be the vertices adjacent to $d$ in $\mathcal{G}^{\pm\dagger}(a,b)$. Now $\mathcal{G}^{\pm\dagger}(a,b)$ is a geodesic, so $c'$ and $c''$ must intersect. Then, either both $c'$ and $c''$ lie in $M$ or both $c'$ and $c''$ lie in $N\setminus M$. We can take an essential curve $c$ from the component which does not contain $c'$ or $c''$ so that $i(c,c')=i(c,c'')=0$. Then, we replace $d$ by $c$ in $\mathcal{G}^{\pm\dagger}(a,b)$. After repeating this process for all curves bounding a M\"obius band in $\mathcal{G}^{\pm\dagger}(a,b)$, we obtain a geodesic in $\mathcal{C}^{\pm\dagger}(N)$ that consists of essential curves. Therefore, $d_{\mathcal{C}^{\dagger}}(a,b)\leq d_{\mathcal{C}^{\pm s}}(a,b)$. 

Next, we prove that $f_{5}$ is 1-dense. Here, it is enough to consider the curves $a$ bounding a M\"obius band in $\mathcal{C}^{\pm\dagger}(N)$. In this case, we take the core curve $b\in\mathcal{C}^{\dagger}(N)$ of the M\"obius band, and we have $d_{\mathcal{C}^{\pm\dagger}}(a,f_{5}(b))=d_{\mathcal{C}^{\pm\dagger}}(a,b)=1$.
\end{proof}

\subsection{Uniform hyperbolicity of fine curve graphs}

We are now ready to prove Theorem~\ref{unif_hyp_fine_curve_graph}. 
Due to Proposition~\ref{fine_curve_graphs_are_quasi-isometric}, it is enough to show the following proposition to prove Theorem~\ref{unif_hyp_fine_curve_graph}.

\begin{proposition}\label{fine_curve_graphs_are_unif_hyp}
For every closed nonorientable closed surface $N=N_{g}$ of genus $g\geq 3$, the extended two-sided fine curve graphs $\mathcal{C}^{\pm\dagger}_{\mathrm{two}}(N)$ is Gromov-hyperbolic. Moreover, its hyperbolicity constant can be taken to be independent of $g$.
\end{proposition}

We need three lemmas to prove Proposition~\ref{fine_curve_graphs_are_unif_hyp}. These lemmas come from Bowden, Hensel, and Webb~\cite[Lemmas 3.6, 3.7, 3.8]{Bowden--Hensel--Webb22} but with the assumption of orientable surfaces changed to nonorientable ones. Instead of our giving the proofs of Lemmas~\ref{distances_are_equivalent},~\ref{distance_of_fine_curve_graph}, and~\ref{distance_of_surviving_curve_graph}, readers may refer to the proofs of~\cite[Lemmas 3.6, 3.7, 3.8]{Bowden--Hensel--Webb22}.

Let $N=N_{g}$ be a closed nonorientable surface of genus $g\geq 3$ and let $P\subset N$ be a finite set. Recall that, for a curve $a$ in $N$, let $[a]_{N-P}$ denote the isotopy class of $a$ in $N-P$. 
Since we treat both a closed surface $N$ and its punctured surface $N-P$, we do not omit surfaces from the symbols for distances in curve graphs.
For instance, we write $d_{\mathcal{C}^{\pm s}_{\mathrm{two}}(N-P)}(\cdot,\cdot)$ for the metric of $\mathcal{C}^{\pm s}_{\mathrm{two}}(N-P)$.

\begin{lemma}{\rm(cf.} \cite[Lemma 3.6]{Bowden--Hensel--Webb22}{\rm)}\label{distances_are_equivalent}
Let $N=N_{g}$ be a closed nonorientable surface of genus $g\geq 3$. Suppose that vertices $a$ and $b$ in $\mathcal{C}^{\pm\dagger}_{\mathrm{two}}(N)$ are transverse, and are in minimal position in $N\setminus P$, where $P\subset N$ is finite and disjoint from $a$ and $b$. Then,
\begin{equation*}
d_{\mathcal{C}^{\pm s}_{\mathrm{two}}(N-P)}([a]_{N-P}, [b]_{N-P})=d_{\mathcal{C}^{\pm\dagger}_{\mathrm{two}}(N)}(a,b).
\end{equation*}
\end{lemma}

\begin{lemma}{\rm(cf.} \cite[Lemma 3.7]{Bowden--Hensel--Webb22}{\rm)}\label{distance_of_fine_curve_graph}
Let $N=N_{g}$ be a closed nonorientable surface of genus $g\geq 3$. Suppose that $a_{1},\cdots,a_{n}$ are two-sided curves including curves bounding a M\"obius band that are pairwise in minimal position in $N-P$. Let $b_{1},\cdots,b_{m}$ be two-sided curves including curves bounding a M\"obius band that are disjoint from $P$. Then, the $b_{i}$ can be isotoped in $N-P$ such that $a_{1},\cdots,a_{n},b_{1},\cdots,b_{m}$ are pairwise in minimal position in $N-P$.
\end{lemma}

\begin{lemma}{\rm(cf.} \cite[Lemma 3.8]{Bowden--Hensel--Webb22}{\rm)}\label{distance_of_surviving_curve_graph}
Let $N=N_{g}$ be a closed nonorientable surface of genus $g\geq 3$. Let $a,b\in\mathcal{C}^{\pm\dagger}_{\mathrm{two}}(N)$ and $P\subset N$ be a finite set. Then, we can find a geodesic $a=\nu_{0},\cdots,\nu_{k}=b$ such that $\nu_{i}\cap P=\emptyset$ for all $0<i<k$.
\end{lemma}

\begin{remark} \label{rmk:on_lemmas}
We assume that curves are two-sided, including those bounding the M\"obius band in the statement in Lemmas~\ref{distances_are_equivalent},~\ref{distance_of_fine_curve_graph}, and~\ref{distance_of_surviving_curve_graph} because it is a convenient assumption with which to prove Proposition~\ref{fine_curve_graphs_are_unif_hyp}; these lemmas also hold for the fine curve graphs $\mathcal{C}^{\dagger}(N)$ and surviving curve graphs $\mathcal{C}^{s}(N)$.
\end{remark}

Now we can prove Proposition~\ref{fine_curve_graphs_are_unif_hyp}.

\begin{proof}[Proof of Proposition~\ref{fine_curve_graphs_are_unif_hyp}]
We will prove that, for all $u,a,b,c\in\mathcal{C}^{\pm\dagger}_{\mathrm{two}}(N)$,
\begin{equation}\label{hyperbolicity_formula}
\langle a,c\rangle_{u}\geq\mathrm{min}\{\langle a,b\rangle_{u},\langle b,c\rangle_{u}\}-\delta-4,
\end{equation}
where $\delta$ is the uniform hyperbolicity constant of $\mathcal{C}^{\pm s}_{\mathrm{two}}(N)$ in Corollary~\ref{extended_two-sided_cur_graph_is_uniformly_hyperbolic}.

We relate the vertices $u,a,b,c\in\mathcal{C}^{\pm\dagger}_{\mathrm{two}}(N)$ with the vertices $u',a',b',c'\in\mathcal{C}^{\pm s}_{\mathrm{two}}(N-P)$ satisfying the assumptions of Lemma~\ref{distances_are_equivalent}, that is, the two properties:

\begin{itemize}
\item[(i)] $d_{\mathcal{C}^{\pm\dagger}_{\mathrm{two}}(N)}(a,a')\leq 1$, $d_{\mathcal{C}^{\pm\dagger}_{\mathrm{two}}(N)}(b,b')\leq 1$, and $d_{\mathcal{C}^{\pm\dagger}_{\mathrm{two}}(N)}(c,c')\leq 1$,
\item[(ii)] the vertices $u'$, $a'$, $b'$, $c'$ are transversal.
\end{itemize}

Set $u' = u$. Find $a''$ that is disjoint from and isotopic to $a$ (note that the two-sidedness of curves is needed in this step); then, find a small enough perturbation $a'$ of $a''$ satisfying (ii) (do this for $b$ and $c$ as well).

Now, choose a finite set $P\subset N$ so that any bigon between a pair from $u'$, $a'$, $b'$, $c'$ contains a point of $P$. Then, by the bigon criterion, this ensures that $u'$, $a'$, $b'$, and $c'$ are pairwise in minimal position in $N-P$.

By Lemma~\ref{distances_are_equivalent}, for any pair $d,e\in\{a,b,c\}$ we have that
\begin{equation}\label{distance_formula}
d_{\mathcal{C}^{\pm\dagger}_{\mathrm{two}}(N)}(d',e')=d_{\mathcal{C}^{\pm s}_{\mathrm{two}}(N-P)}([d']_{N-P},[e']_{N-P}),
\end{equation}
and
\begin{equation}\label{Gromov_product_formula}
|\langle d',e'\rangle_{u'}-\langle d,e\rangle_{u}|\leq 2.
\end{equation}
Finally, we obtain formula (\ref{hyperbolicity_formula}) from (\ref{distance_formula}), (\ref{Gromov_product_formula}) above, and Corollary~\ref{extended_two-sided_cur_graph_is_uniformly_hyperbolic}.
\end{proof}

\subsection{Low genus cases}\label{Low_genus_cases}

In this subsection, we discuss the Gromov hyperbolicity of the fine curve graph $\mathcal{C}^{\dagger}(N_g)$ for $g=2$. 
Through Subsection~\ref{Low_genus_cases}, let $N=N_{g,n}$ be a nonorientable surface of genus $g\leq 2$ with $n\geq 0$ punctures.
For nonorientable surfaces $N$ of genus $g=1,2$, we modify the definition of the fine curve graph $\mathcal{C}^{\dagger}(N)$ of $N$ so that two vertices form an edge if the corresponding curves intersect at most once, and the surviving curve graph $\mathcal{C}^{s}(N)$ of $N$ so that two vertices form an edge if we can choose the representatives which intersect at most once.
Furthermore, the ordinary curve graph $\mathcal{C}(N)$ and the nonseparating curve graph $\mathcal{NC}(N)$ of $N$ of genus $g=1,2$ are also defined in the same way as $\mathcal{C}^{s}(N)$.
We see that $\mathcal{NC}(N)$ of $N$ of genus $g=1,2$ is path-connected (see~\cite{Kuno21}).

Firstly, we obtain the following:
\begin{lemma}
Let $N=N_{g,n}$ be a nonorientable surface of genus $g\leq 2$ with $n\geq 0$ punctures.
Then we have $\mathcal{C}^{s}(N)=\mathcal{NC}(N)$.
\end{lemma}

\begin{proof}
We note that any nonseparating curve is surviving; hence we have $\mathcal{NC}(N)\subset\mathcal{C}^{s}(N)$.
We will show $\mathcal{C}^{s}(N)\subset\mathcal{NC}(N)$.
We remark that any one-sided curve is surviving.
In the case where $g=1$ and $n=0,1$, there is only one isotopy class of essential simple closed curves, and it is one-sided, and so nonseparating. In particular, we see $\mathcal{C}^{s}(N)\subset\mathcal{NC}(N)$.
In the case where $g=1$ and $n\geq 2$, let $c$ be a surviving curve on $N$. Since surviving curves must essential on the surface which we fill in all punctures on $N$, $c$ must be one-sided, and so it is nonseparating.
In the case where $g=2$ and $n=0,1$, there are no essential separating curves on $N_{2,0}$ and $N_{2,1}$.
Hence if $c$ is a surviving curve on $N_{2,0}$ and $N_{2,1}$, then particularly $c$ is essential and it is nonseparating.
In the case where $g=2$ and $n\geq 2$, to prove the contrapositive, let $c$ be a separating curve on $N_{2,n}$.
Then $c$ separates $N$ to two connected components $F_{1}$ and $F_{2}$.
If $F_{1}$ contains the two M\"obius bands, then $F_{2}$ is a punctured disk, and $c$ is not surviving.
If $F_{1}$ contains exactly one crosscap, then $F_{2}$ also contains exactly one crosscap, and $c$ is not surviving.
\end{proof}

The nonseparating curve graphs $\mathcal{NC}(N)$ are uniformly hyperbolic for $g=1$ and $n\geq 2$, and for $g=2$ and $n\geq 1$ \cite{Kuno21}.
The other cases, that is, $g=1$ and $n\leq 1$, or $g=2$ and $n=0$,  the nonseparating curve graph $\mathcal{NC}(N)$ is finite and path-connected, and so it is hyperbolic.
Then we have the following corollary:

\begin{corollary}\label{hyperbolicity_surviving_curve_graph_low_genera}
For every nonorientable surface $N=N_{g,n}$ of genus $g\leq 2$ with $n\geq 0$ punctures, 
$\mathcal{C}^{s}(N)$ is Gromov-hyperbolic.
Moreover, its hyperbolicity constant can be taken to be independent of $g$ and $n$.

\end{corollary}

\begin{lemma}
Let $N=N_{g,n}$ be a nonorientable surface of genus $g=1,2$ with $n$ punctures.
Then, $\mathcal{C}(N)$ is path-connected.
\end{lemma}

\begin{proof}
Let $a$ and $b$ be two vertices of $\mathcal{C}(N)$ which intersect. In this proof, slightly abusing the notation we also consider the vertices of $\mathcal{C}(N)$ as their representative curves on $N$.
First we consider the cases where $g=1$ and $n\geq 2$, or $g=2$ and $n\geq 1$. Since $\mathcal{NC}(N)$ is path-connected, we assume that at least one of $a$ and $b$ is separating. In the case where $g=1$, if $a$ is a separating curve, then the exactly one component $M$ of $N\setminus a$ contain the unique M\"obius band. Let $a'$ be a one-sided curve in $M$. Then $a'$ is disjoint from $a$. If $b$ is also separating, then we can also take a one-sided curve $b'$ which is disjoint from $b$ by the same way. Since both $a'$ and $b'$ are nonseparating, we can connect $a'$ and $b'$ by a path in $\mathcal{NC}(N)$. Therefore we can also connect $a$ and $b$ by a path in $\mathcal{C}(N)$. In the case where $g=2$, we use a similar argument above and we can prove $\mathcal{C}(N)$ is path-connected.

Second, we consider the case where $g=1$ and $n\leq 1$, or $g=2$ and $n=0$.
In these cases, $\mathcal{C}(N)=\mathcal{NC}(N)$, and hence we see that $\mathcal{C}(N)$ is path-connected. 

\end{proof}

\begin{lemma}\label{path-conn_extended_curve_graph_low_genus}
Let $N=N_{2,n}$ be a nonorientable surface of genus $2$ with $n$ punctures.
Then, $\mathcal{C}^{\pm}(N)$ is path-connected.
\end{lemma}

\begin{proof}
Let $a$ and $b$ be two vertices of $\mathcal{C}^{\pm}(N)$ which intersect. If both $a$ and $b$ does not bound a M\"obius band, then we can connect $a$ and $b$ by a path in $\mathcal{C}(N)$. We assume that at least one of $a$ and $b$ bounds a M\"obius band. If $a$ bounds a M\"obius band $M$, then we can take an essential curve $a'$ isotopic to the core curve of $M$ such that $i(a,a')=0$. If $b$ bounds a M\"obius band $M'$, we can also take  an essential curve $b'$ isotopic to the core curve of $M'$ such that $i(b,b')=0$. Since $a'$ and $b'$ are both essential, we can connect $a'$ and $b'$ by a path in $\mathcal{C}(N)$. Hence we can also connect $a$ and $b$ by a path in $\mathcal{C}^{\pm}(N)$.
\end{proof}

\begin{lemma}
Let $N=N_{2,n}$ be a nonorientable surface of genus $2$ with $n$ punctures.
Then, $\mathcal{C}^{\dagger}(N)$ and $\mathcal{C}^{\pm\dagger}(N)$ are path-connected.
\end{lemma}

\begin{proof}
We can prove by the same argument as that of Proposition~\ref{fine_curve_graphs_are_path_connected} for $\mathcal{C}^{\dagger}(N)$ and $\mathcal{C}^{\pm\dagger}(N)$ of genus $g\geq 3$ by using the path connectedness of $\mathcal{C}^{\pm}(N)$ (Lemma~\ref{path-conn_extended_curve_graph_low_genus}).
\end{proof}

We remark that the properties corresponding to Lemmas~\ref{distances_are_equivalent},~\ref{distance_of_fine_curve_graph}, and~\ref{distance_of_surviving_curve_graph} also hold for the fine curve graph $\mathcal{C}^{\dagger}(N)$ and the surviving curve graph $\mathcal{C}^{s}(N)$ of a closed nonorientable surface of genus 2.
Thus we obtain the following result.

\begin{theorem}\label{hyperbolicity_fine_curve_graph_genus_two}
Let $N=N_{2}$ be a closed nonorientable surface of genus $2$, and 
$\delta$ be the hyperbolicity constant of $\mathcal{C}^{s}(N)$ in Corollary~$\ref{hyperbolicity_surviving_curve_graph_low_genera}$.
Then, the fine curve graph $\mathcal{C}^{\dagger}(N)$ is $(\delta+4)$-hyperbolic.
\end{theorem}

The proof of Theorem~\ref{hyperbolicity_fine_curve_graph_genus_two} is the same as the proof of Proposition~\ref{fine_curve_graphs_are_unif_hyp}, but the condition (i) in the proof holds also for the one sided curves since we can connect two vertices by one edge when the corresponding curves intersect at most once for genus $g=1,2$.

\subsection{Topological fine curve graph}\label{Topological fine curve graph}

A {\it topological fine curve graph} $\mathcal{C}^{\dagger}(N)_\mathrm{top}$ is the graph whose vertex is a topologically-embedded essential simple closed curve and two vertices form an edge if the corresponding curves are disjoint. In this subsection, we prove that the uniform hyperbolicity of $\mathcal{C}^{\dagger}(N)_\mathrm{top}$ for $N=N_{g}$ ($g\geq 3$).

By the same argument in the proofs of Proposition~\ref{fine_curve_graphs_are_path_connected} and Proposition~\ref{fine_curve_graphs_are_quasi-isometric}, we have the following.
\begin{proposition}\label{top_fine_curve_graphs_are_quasi-isometric}
Let $N=N_{g,p}$ be a nonorientable surface of genus $g\geq 3$ with $p\geq 0$ punctures. Then, $\mathcal{C}^{\dagger}(N)_\mathrm{top}$, $\mathcal{C}^{\pm \dagger}(N)_\mathrm{top}$, and $\mathcal{C}^{\pm \dagger}_{\mathrm{two}}(N)_\mathrm{top}$ are path-connected, and
\begin{equation*}
 \mathcal{C}^{\dagger}(N)_\mathrm{top} \qisim \mathcal{C}^{\pm \dagger}(N)_\mathrm{top} \qisim \mathcal{C}^{\pm\dagger}_{\mathrm{two}}(N)_\mathrm{top}.
\end{equation*}
Moreover, we can take quasi-isometries between them so that their quasi-isometric constants do not depend on $g$ and $n$.    
\end{proposition}

Similar to~\cite[Lemma 3.11]{Bowden--Hensel--Webb22} we also obtain the following lemma:
\begin{lemma}\label{top_fine_curve_graph_is_qi_to_fine_curve_graph}
Let $N=N_{g}$ be a nonorientable surface of genus $g\geq 3$. Then, the inclusion map $\mathcal{C}^{\pm\dagger}_{\mathrm{two}}(N)\hookrightarrow\mathcal{C}^{\pm\dagger}_{\mathrm{two}}(N)_\mathrm{top}$ is $(1,2)$-quasi-isometric embedding and $1$-dense.

\end{lemma}

Moreover, for a closed nonorientable surface of genus 2, we see the following:
\begin{lemma}\label{top_fine_curve_graph_is_qi_to_fine_curve_graph_genus_two}
Let $N=N_{2}$ be a nonorientable surface of genus $2$. Then, the inclusion map $\mathcal{C}^{\dagger}(N)\hookrightarrow\mathcal{C}^{\dagger}(N)_\mathrm{top}$ is $(1,2)$-quasi-isometric embedding and $1$-dense.
\end{lemma}

The proofs of Lemmas~\ref{top_fine_curve_graph_is_qi_to_fine_curve_graph} and~\ref{top_fine_curve_graph_is_qi_to_fine_curve_graph_genus_two} are the same as that of Bowden, Hensel, and Webb~\cite[Lemma 3.11]{Bowden--Hensel--Webb22}, so we omit the proofs.

By Proposition~\ref{top_fine_curve_graphs_are_quasi-isometric}, Lemmas~\ref{top_fine_curve_graph_is_qi_to_fine_curve_graph} and~\ref{top_fine_curve_graph_is_qi_to_fine_curve_graph_genus_two}, and the uniformly hyperbolicity of $\mathcal{C}^{\pm\dagger}_{\mathrm{two}}(N)$, we obtain the following result:
\begin{corollary}
Let $N=N_{g}$ be a nonorientable closed surface of genus $g\geq 2$.
Then, the topological fine curve graph $\mathcal{C}^{\dagger}(N)_\mathrm{top}$ is Gromov-hyperbolic.
Moreover, its hyperbolicity constant can be taken to be independent of $g$. 
\end{corollary}

\section{Construction of quasimorphisms} \label{sec:qm}

In this section, we prove Theorems \ref{thm:main} and \ref{thm:homeo}.
Our strategy for proving these theorems is the same as for the case of orientable surfaces \cite{Bowden--Hensel--Webb22}.

\begin{lemma} \label{lem:hyp}
 
 Let $N$ be a closed nonorientable surface and $P \subset N$ be a set of finite points. Let $\psi \in {\rm Diff}(N)$ be a diffeomorphism such that $\varphi(P)=P$ and set $f=[\psi]_{N_g-P} \in {\rm MCG}(N-P)$. If $f$ is pseudo-Anosov, then $\psi$ is a hyperbolic element with respect to the action on $\mathcal{C}^{\dagger}(N)$.
\end{lemma}

\begin{proof}
  Let $a,b$ be two vertices in $\mathcal{C}_{\rm two}^{\pm \dagger}(N)$ such that the curves $a,b$ are in $N - P$.
  By using Lemma \ref{distance_of_surviving_curve_graph} (see also Remark \ref{rmk:on_lemmas}), we can obtain that
  \begin{equation} \label{eq:qi_3}
    d_{\mathcal{C}^{s}(N-P)}([a]_{N-P},[b]_{N-P}) \leq d_{\mathcal{C}^{\dagger}(N)}(a,b).
  \end{equation}
  Thus, by inequality \eqref{eq:qi_3} and since
  the inclusion map $\mathcal{C}^s(N-P) \to \mathcal{C}(N-P)$ is 1-Lipschitz,
  \begin{equation*} 
      d_{\mathcal{C}(N-P)}([a]_{N-P},f^k( [a]_{N-P})) \leq  d_{\mathcal{C}^{\dagger}(N)}(a,\psi^k(a))
  \end{equation*}
  for every $k \in \mathbb{Z}$. Note that $f^k ([a]_{N-P})= [\psi^k(a)]_{N-P}$.
Hence we obtain $ |\varphi|_{\mathcal{C^\dagger}(N)} \geq |f|_{\mathcal{C}(N-P)}$.
Since $f$ is pseudo-Anosov, 
$|f|_{\mathcal{C}(N-P)}>0$ (see \cite{Bestvina--Fujiwara07}).
Thus we complete the proof.
\end{proof}

\begin{lemma} \label{lem:pA}
  Let $N$ be a closed nonorientable surface of genus $g \geq 3$ and $p \in N$ be a point. Then there exists $\psi \in {\rm Diff}_0(N)$ such that $\psi(p)=p$ and $[\psi]_{N-p}$ is pseudo-Anosov.
\end{lemma}

\begin{proof}
  Let $\pi \colon S\to N$ be the double covering. Let $\gamma$ be a filling curve in $N$ based at $p$.
  We can assume that $\pi^{-1}(\gamma)$ is connected (by adding a based loop at $p$ to $\gamma$ if necessary).
  Then $\tilde{\gamma}=\pi^{-1}(\gamma)$ is a filling curve in $S$.
  Let $\psi \in {\rm Diff}_0(N)$ be the point pushing map of $p$ along $\gamma$, and $\tilde{\psi} \in {\rm Diff}_0(S)$ the lift of $\psi$. 
  Since $\chi(S)<0$ and by Theorem \ref{thm:IM}, $[\tilde{\psi}]_{S-\pi^{-1}(p)}$ is pseudo-Anosov and $[\psi]_{N-p}$ is also.
\end{proof}

\begin{proof}[Proof of Theorem \ref{thm:main}]
  Let ${\rm pr} \colon \mathcal{C}^{\dagger}(N) \to \mathcal{C}(N)$ be the natural projection.
  Let $p \in N$ be a point and $\psi \in {\rm Diff}_0(N)$ a diffeomorphism as in Lemma \ref{lem:pA}.
  Then, by Lemma \ref{lem:hyp}, the diffeomorphism $\psi$ is a hyperbolic element.

  Let $a \in \mathcal{C}^{\dagger}(N)$ be a vertex. For some $K,L \geq 0$, the set $A_{\psi}=\{ \psi^n (a) \}_{n \in \mathbb{N}}$ is a $(K,L)$-quasi-axis. Note that ${\rm pr}(A_{\psi})=\{ [a]_N \}$ is a singleton.

  We take a pseudo-Anosov map $\varphi \in {\rm Diff}(N)$. For every $k \in \mathbb{Z}$, the map $\varphi^k \psi \varphi^{-k} \in {\rm Diff}_0(N)$ is a hyperbolic element.
  Then $A_{\varphi^k \psi \varphi^{-k}}=\{ (\varphi^k \psi \varphi^{-k})^n (\varphi^k(a)) \}_{n \in \mathbb{N}}$ is also a $(K,L)$-quasi-axis and ${\rm pr}(A_{\varphi^k \psi \varphi^{-k}})=\{ \varphi^k ([a]_N) \}$ is also a singleton.

  Then, for a sufficiently large $k \in \mathbb{Z}$,
  \[d_{\mathcal{C}(N)}([a]_N,  [\varphi^k(a)]_N ) > B(K,L,\delta)\]
  since $\varphi$ is pseudo-Anosov. Thus, for any $x \in A_{\psi}$ and $\phi \in {\rm Diff}_0(N)$,
\begin{align*}
    d_{\mathcal{C}^{\dagger}(N)} ( \phi(x), A_{\varphi^k \psi \varphi^{-k} } )
  &\geq  d_{\mathcal{C}(N)}([\phi(a)]_N,  [\varphi^k(a)]_N ) \\
  &= d_{\mathcal{C}(N)}([a]_N,  [\varphi^k(a)]_N ) > B(K,L,\delta).
\end{align*}
This shows that $g_1= \psi$ and $g_2 = \varphi^k \psi \varphi^{-k}$ satisfies the assumption of Proposition \ref{prop:BF}. Therefore, by Proposition \ref{prop:BF} and Theorem \ref{unif_hyp_fine_curve_graph}, we complete the proof.
\end{proof}

\begin{remark}
Although 
the graph $C^{\dagger}(N_2)$ is hyperbolic (Theorem \ref{hyperbolicity_fine_curve_graph_genus_two}), we cannot adapt the argument in the proof of Theorem \ref{thm:main} directly for the genus 2 case since the Klein bottle $N_2$ does not admit any pseudo-Anosov diffeomorphism.
We remark that B\"{o}ke recently proved that the space $\widetilde{QH}(\mathrm{Diff}_0(N_2))$ is non-trivial \cite{Boke}.
\end{remark}

Finally, we consider the case of homeomorphism groups. 
The following theorem is due to Kotschick.
\begin{theorem}[Kotschick] \label{thm:conti}
Let $F$ be a closed surface. 
Then every homogeneous quasimorphism on ${\rm Diff}_0(F)$ is continuous with respect to the $C^0$-topology.
\end{theorem}
See \cite[Appendix A.]{Bowden--Hensel--Webb22} for more details. 
We note that in \cite{Bowden--Hensel--Webb22} they treat the case of orientable surfaces but their argument also goes through for nonorientable surfaces; the discussion is local except for triangulation, and nonorientable surfaces can also be triangulated. Since every diffeomorphism on a closed surface is uniformly approximated by diffeomorphisms, every homogeneous quasimorphism on ${\rm Diff}_0(F)$ extends to ${\rm Homeo}_0(F)$ and we obtain Theorem \ref{thm:homeo}. 
As we mentioned in introduction, Theorem \ref{thm:homeo} can also be proved by using the topological fine curve graph $C^{\dagger}(N)_{\rm top}$ directly.

\section*{Acknowledgment}
The authors wish to express their great appreciation to Genki Omori for valuable comments on mapping class groups of nonorientable surfaces.
The authors are also deeply grateful to Kazuhiro Ichihara for answering their questions about filling curves.
The authors would also like to thank Jonathan Bowden for his comments on the genus 2 case.

The first author is supported by JSPS KAKENHI Grant Number JP20H00114 and JP24K16921, and JST-Mirai Program Grant Number JPMJMI22G1.
The second author is supported by the Foundation of Kinoshita Memorial Enterprise, by a JSPS KAKENHI Grant-in-Aid for Early-Career Scientists, Grant Number 21K13791, and by JST, ACT-X, Grant Number JPMJAX200D.

\end{document}